\documentclass[11pt,reqno]{amsart}
\usepackage{amssymb}
\usepackage{amsmath}
\usepackage{graphicx}
\usepackage{verbatim}
\usepackage{mathtools}
\usepackage{enumitem}
\usepackage[breaklinks]{hyperref}
\usepackage{etex}
\usepackage[dvipsnames,x11names,svgnames]{xcolor} 
\definecolor{celadon}{rgb}{0.67, 0.88, 0.69}
\colorlet{kollane}{LightGoldenrod1}
\colorlet{roheline}{celadon}
\colorlet{sinine}{LightSkyBlue2}
\colorlet{lilla}{violet}
\colorlet{hall}{lightgray}
\usepackage[framemethod=tikz]{mdframed}
\usepackage[framemethod=tikz]{mdframed}
\usepackage{enumitem} 

\newtheorem{thm}{Theorem}[section]
\newtheorem{theorem}[thm]{Theorem}

\newtheorem{lem}[thm]{Lemma}

\newtheorem{prop}[thm]{Proposition}
\newtheorem{proposition}[thm]{Proposition}
\newtheorem*{Claim}{Claim}
\newtheorem{fact}{Fact}[section]
\theoremstyle{definition}
\newtheorem{definition}{Definition}[section]
\newtheorem{defn}[definition]{Definition}
\theoremstyle{remark}
\newtheorem{remark}{Remark}[section]
\newtheorem{rem}[remark]{Remark}
\numberwithin{equation}{section}

\newcommand{\eps}{\varepsilon}

\newcommand{\N}{\mathbb{N}}

\newcommand{\R}{\mathbb{R}}

\newcommand{\Xs}{X^\ast}
\newcommand{\xs}{x^\ast}
\newcommand{\Xss}{X^{\ast\ast}}
\newcommand{\xss}{x^{\ast\ast}}

\newcommand{\ys}{y^\ast}
\newcommand{\yss}{y^{\ast\ast}}

\DeclareMathOperator{\linspan}{span}

\DeclareMathOperator{\ran}{ran}

\DeclareFontFamily{U}{matha}{\hyphenchar\font45}
\DeclareFontShape{U}{matha}{m}{n}{
      <5> <6> <7> <8> <9> <10> gen * matha
      <10.95> matha10 <12> <14.4> <17.28> <20.74> <24.88> matha12
      }{}
\DeclareSymbolFont{matha}{U}{matha}{m}{n}
\DeclareFontSubstitution{U}{matha}{m}{n}

\DeclareFontFamily{U}{mathx}{\hyphenchar\font45}
\DeclareFontShape{U}{mathx}{m}{n}{
      <5> <6> <7> <8> <9> <10>
      <10.95> <12> <14.4> <17.28> <20.74> <24.88>
      mathx10
      }{}
\DeclareSymbolFont{mathx}{U}{mathx}{m}{n}
\DeclareFontSubstitution{U}{mathx}{m}{n}

\DeclareMathDelimiter{\vvvert}{0}{matha}{"7E}{mathx}{"17}
\newcommand{\Vvert}{\vvvert}
\begin{document}

\title[Strictly convex renormings and the diameter $2$ property]%
{Strictly convex renormings and the diameter $2$ property}
\author[Nygaard]{Olav Nygaard}
\author[P\~{o}ldvere]{M\"{a}rt P\~{o}ldvere}
\author[Troyanski]{Stanimir Troyanski}
\author[Viil]{Tauri Viil}

\address{Department of Mathematics, University of Agder, Servicebox 422,
4604 Kristiansand, Norway.}
\email{olav.dovland@uia.no}

\address{Institute of Mathematics and Statistics, University of Tartu, Narva mnt~18,
51009 Tartu, Estonia.}
\email{mart.poldvere@ut.ee, tauri.viil@gmail.com}

\address{Institute of Mathematics and Informatics,
Bulgarian Academy of Science, bl.~8,
acad. G. Bonchev str., 1113 Sofia, Bulgaria,
and
Departamento de Matem\'{a}ticas, Universidad de Murcia,
Campus de Espinardo, 30100 Espinardo (Murcia), Spain
}
\email{stroya@um.es}
\thanks{The second named author was supported by the Estonian Research Council grant PRG1901.}
\thanks{The third named author was supported by  
Bulgarian National Scientific Fund, Grant, KP–06–H22/4, 04.12.2018, 
and by grant PID2021-122126NB-C32 of Agencia Estatal de Investigaci\'{o}n and EDRF/FEDER  “A way of making Europe”    
(MCIN/AEI 10.13039/501100011033)}
%
\begin{abstract}
A Banach space (or its norm) is said to have the diameter~$2$ property (D$2$P in short)
if every nonempty relatively weakly open subset of its closed unit ball has diameter~$2$.
We construct an equivalent norm on $L_1[0,1]$ which is weakly midpoint locally uniformly rotund and has the D$2$P.
We also prove that for Banach spaces admitting a norm-one finite-co-dimensional projection
it is impossible to be uniformly rotund in every direction and at the same time  have the D$2$P.
\end{abstract}
\maketitle

\section{Introduction}
The purpose of this paper is to study the implications of the ``big weakly open set in the unit ball phenomena''
on rotundity properties of the norm.

\begin{defn} \label{defn:diameter 2 properties}
A Banach space $X$ (or its norm) has the \emph{diameter $2$ property} (D$2$P for short)
if every nonempty relatively weakly open subset of its closed unit ball $B_X$ has diameter $2$.
\end{defn}

Intuitively, the D$2$P should prevent a Banach space from having ``too good'' rotundity properties.
For instance, in \cite[Theorem 2.7]{OST}, the existence of strictly convex equivalent norms on a Banach space
has been characterised in terms of the existence of countably many families of slices
which separate the points of its unit sphere.

\begin{defn}\label{defn:MLUR}
A Banach space $X$ (or its norm) is \emph{midpoint locally uniformly rotund} (MLUR for short)
(resp. \emph{weakly midpoint locally uniformly rotund} (wMLUR for short))
if every $x$ in the unit sphere $S_X$ of $X$ is a strongly extreme point
(resp. strongly extreme point in the weak topology) of $B_X$,
i.e., for every sequence $(x_n)_{n=1}^\infty$ in $X$ such that $\|x \pm x_n\|\xrightarrow[n\to\infty]{}1$,
one has $x_n\xrightarrow[n\to\infty]{}0$ in norm (resp. weakly).
\end{defn}

It is not obvious that wMLUR spaces (and even strictly convex spaces) with the D$2$P exist, but indeed they do.
The quotient $C(\mathbb{T})/A$,
where $C(\mathbb{T})$ is the space of continuous functions on the complex unit circle $\mathbb{T}$
and $A$ is the disc algebra,
is such an example
(this space is wMLUR and is $M$-embedded, see \cite[pages 167--168, Remark IV.1.17]{HWW}).
In \cite{AHNTT} an equivalent MLUR norm on $C[0,1]$ with the D$2$P was constructed.
An open question, already posed by V.~Kadets in \cite[Section~5]{K},
is whether there exists a strictly convex Banach space with the Daugavet property.

Clearly Banach spaces with the Radon--Nikod\'{y}m property fail to have the D$2$P.
The space $L_1[0,1]$ has the analytic Radon--Nikod\'{y}m property and the D$2$P.
In this paper, we construct an equivalent strictly convex norm with the D$2$P on $L_1[0,1]$.
To our knowledge, this is the first example of a strictly convex Banach space with the D$2$P
which does not have subspaces isomorphic to $c_0$.
Namely, we define an equivalent norm $\Vvert\cdot\Vvert$ on $L_1[0,1]$ by the formula
\begin{equation}\label{L1-renorming-formula}
\Vvert f \Vvert = \left(\sum_{k=0}^\infty 4^{-k} \sum_{j=1}^{2^k}\|f\|_{k,j}^2\right)^{\frac{1}{2}}
\end{equation}
where, setting $I_j^k=\bigl[\frac{j-1}{2^k},\frac{j}{2^k}\bigr)$
and letting $\lambda$ be the Lebesgue measure on~$[0,1]$,
\begin{equation*}
\| f \|_{k,j} = \int_{I_j^k}|f|\,d\lambda.
\end{equation*}

Our main result (Propositions~\ref{prop:rommet_har_D2P} and \ref{prop:Rommet_er_wMLUR} in tandem) is the following theorem.

\begin{thm}\label{thm: (L_1[0,1],|||.|||) is wMLUR and has the D2P}
The Banach space $(L_1[0,1],\Vvert\cdot\Vvert)$ is wMLUR and has the D$2$P.
\end{thm}

\begin{defn}\label{defn: URED}
A Banach space $X$ (or its norm) is \emph{uniformly rotund in every direction} (URED for short)
if whenever $(x_n)_{n=1}^\infty$ and $(y_n)_{n=1}^\infty$ are two sequences in $B_X$
such that there are $z\in X$ and scalars $r_n$, $n=1,2,\dotsc$, with $x_n-y_n = r_nz$ for every $n\in\N$,
and $\| x_n + y_n \|\xrightarrow[n\to\infty]{}2$,
one has $\| x_n - y_n \|\xrightarrow[n\to\infty]{}0$.
\end{defn}

It is natural to ask whether there exist URED Banach spaces with the D$2$P.
This question remains open, but we have the following result towards the negative
from which \cite[Proposition~2.10]{ALNT} is a corollary.

\begin{thm}\label{thm: norm-one projection with finite-co-dimensional range, URED, and D2P}
Let $X$ be a Banach space on which there exists a norm-one projection $P$
with finite-co-dimensional range.
If $X$ has the D$2$P, then $X$ is not URED and $\Xss$ is not strictly convex.
Moreover, whenever $0<\delta<1$, the unit sphere of $\Xss$ contains a line segment of length $1-\delta$.
\end{thm}

We use standard Banach space notation.
The letters $\mathcal{L}$ and $\lambda$ will stand, respectively,
for the $\sigma$-algebra of Lebesgue measurable subsets of the interval $[0,1]$
and for the Lebesgue measure on $\mathcal{L}$.
By $\|\cdot\|_1$ and $\|\cdot\|_\infty$ we denote the canonical norms on $L_1[0,1]$ and $L_\infty[0,1]$, respectively.
For a subset $A$ of the interval $[0,1]$, the symbol ${\bf 1}_A$ will stand for its characteristic function.
We deal only with real Banach spaces.

A very recent overview of renorming techniques can be found in \cite{GMZ}.

\section{Proof of Theorem \ref{thm: (L_1[0,1],|||.|||) is wMLUR and has the D2P}}

First things first:
since $\|\cdot\|_{0,1}=\|\cdot\|_1$ and $\|\cdot\|_{k,j}\leq\|\cdot\|_1$
whenever $k\in\N$ and $j\in\{1,\dotsc,2^k\}$,
the norm $\Vvert\cdot\Vvert$ defined by \eqref{L1-renorming-formula}
is equivalent to the original norm $\|\cdot\|_1$ on $L_1[0,1]$
(with $\|\cdot\|_1\leq\Vvert\cdot\Vvert\leq\sqrt{2}\|\cdot\|_1$).

For convenience of reference, let us point out two facts on the intervals~$I^k_j$.

\begin{fact}\label{fact2.1}
Whenever $k,m\in\{0\}\cup\N$ with $k<m$, and $j\in\{1,\dotsc,2^k\}$,
the interval $I_j^k$ is equal to the (disjoint) union $\bigcup_{i=(j-1)\cdot 2^{m-k}+1}^{j\cdot 2^{m-k}}I^m_i$.
\end{fact}

\begin{fact}\label{fact2.2}
Whenever $A\in\mathcal{L}$ and $\eps>0$,
there are $k\in\{0\}\cup\N$ and a subset $J\subset\{1,\dotsc,2^k\}$
such that $\lambda\Bigl(A\triangle \bigcup_{j\in J}I^k_j\Bigr)<\eps$.
\textup{It follows that} the linear span of the set $\bigl\{{\bf 1}_{I^k_j}\colon k\in\{0\}\cup\N,\, j\in\{1,\dotsc,2^k\}\bigr\}$
is norm dense in $L_1[0,1]$.
\end{fact}

The following proposition is the first portion of Theorem \ref{thm: (L_1[0,1],|||.|||) is wMLUR and has the D2P}.

\begin{prop}\label{prop:rommet_har_D2P}
The Banach space $X=(L_1[0,1],\Vvert\cdot\Vvert)$ has the D$2$P.
\end{prop}
\begin{proof} Let $W$ be a weakly open set in $X$ which intersects the closed unit ball of $X$.
Since (essentially) bounded functions form a norm dense subspace in $L_1[0,1]$,
there exists a bounded on $[0,1]$ function $f\in W$ with $\Vvert f \Vvert =1$.
By identifying the dual space of $L_1[0,1]$ with $L_\infty[0,1]$ in the canonical way,
it follows that there are $m\in\N$, $h_1,\dotsc,h_m \in L_\infty [0,1]$
with $\|h_l\|_\infty=1$ for every $l\in\{1,\dotsc,m\}$,
and $\delta>0$ such that
\[
V:=\bigl\{g\in L_1[0,1]\colon \text{$|\langle g-f,h_l\rangle|<\delta$ for every $l\in\{1,\dotsc,m\}$}\bigr\}\subset W.
\]
Set $U=V\cap B_{X}$ and let $\eps>0$ be arbitrary.
In order to finish the proof it suffices to find  $g_1,g_2\in U$ such that $\Vvert g_1-g_2\Vvert>2-\varepsilon$.

Pick $\gamma\in (0,1)$ such that
\begin{equation}\label{one}
(5\|f\|_\infty+1)\gamma<\delta
\qquad\text{and}\qquad
2(1-\gamma)^{\frac{3}{2}}>2-\varepsilon.
\end{equation}
According to Facts~\ref{fact2.2} and \ref{fact2.1}, there are $K\in\mathbb{N}$
and functions $s_1,\dotsc,s_m\in \linspan\bigr\{{\bf 1}_{I_j^K}\colon\allowbreak j\in\{1,2,\dotsc 2^K\}\bigr\}$ such that
\begin{align}\label{two}
2^{-K}<\gamma
\end{align}
and
\begin{align}\label{three}
\|h_l-s_l\|_1&<\gamma \quad\text{for every  $l\in\{1,\dotsc,m\}$.}
\end{align}
For every $j\in\{1,2,\dotsc,2^K\}$, set
\[
b_j:=\int_{I_j^K}f^+\,d\lambda
\qquad\text{and}\qquad
c_j:=\int_{I_j^K}f^-\,d\lambda;
\]
we then have
\begin{align}
\int_{I_j^K}f\,d\lambda
&=\int_{I_j^K}(f^+ -f^-)\,d\lambda
=b_j-c_j,\label{eq: int_(I^k_j)f dlambda=b_j-c_j}\\
\int_{I_j^K}|f|\,d\lambda
&=\int_{I_j^K}(f^+ +f^-)\,d\lambda
=b_j+c_j.\label{eq: int_(I^k_j)|f|dlambda=b_j+c_j}
\end{align}
Define
\[
f_1
:=2^{K+2}\sum_{j=1}^{2^K}\left(b_j {\bf 1}_{I_{4j-3}^{K+2}}-c_j {\bf 1}_{I_{4j-2}^{K+2}}\right)
\]
and
\[
f_2
:=2^{K+2}\sum_{j=1}^{2^K}\left(b_j {\bf 1}_{I_{4j-1}^{K+2}}-c_j {\bf 1}_{I_{4j}^{K+2}}\right).
\]
It is straightforward to verify that
whenever $k\in\{0,\dotsc,K\}$ and $j\in\{1,\dotsc,\allowbreak2^k\}$, one has
\begin{align}
\int_{I^k_j}f_1\,d\lambda
&=\int_{I^k_j}f_2\,d\lambda=\int_{I^k_j}f\,d\lambda,\label{five}\\
\int_{I^k_j}|f_1|\,d\lambda
&=\int_{I^k_j}|f_2|\,d\lambda=\int_{I^k_j}|f|\,d\lambda,\label{six}\\
\int_{I^k_j}|f_1-f_2|\,d\lambda
&=2\int_{I^k_j}|f|\,d\lambda.\label{seven}
\end{align}
Indeed, \eqref{five}--\eqref{seven} for $k=K$ follow
from \eqref{eq: int_(I^k_j)f dlambda=b_j-c_j} and \eqref{eq: int_(I^k_j)|f|dlambda=b_j+c_j}
by observing that $I_j^K$ is the disjoint union
of the intervals $I_{4j-3}^{K+2}$, $I_{4j-2}^{K+2}$, $I_{4j-1}^{K+2}$, $I_{4j}^{K+2}$
of equal length $2^{-(K+2)}$.
The results for $k\in\{0,1,2,\dotsc,K-1\}$ follow directly from those for $k=K$
when we bring to mind Fact~\ref{fact2.1}.

We rewrite \eqref{six} and \eqref{seven} as
\begin{equation}\label{eight}
\|f_i\|_{k,j}
=\|f\|_{k,j}
\quad\text{whenever $i\in\{1,2\}$, $k\in\{0,\dotsc,K\}$, and $j\in\{1,\dotsc,2^k\}$,}
\end{equation}
and
\begin{equation}\label{nine}
\|f_1-f_2\|_{k,j}
=2\|f\|_{k,j}
\quad\text{whenever $k\in\{0,\dotsc,K\}$ and $j\in\{1,\dotsc,2^k\}$.}
\end{equation}

\noindent%
In particular, we get for each $i\in\{1,2\}$,
\begin{equation}\label{ten}
\|f_i\|_1
=\|f_i\|_{0,1}
=\|f\|_{0,1}
=\|f\|_1
\leq \Vvert f\Vvert
=1
\end{equation}
and therefore, whenever $i\in\{1,2\}$, $k\in\{0,\dotsc,K\}$, and $j\in\{1,\dotsc,2^k\}$,
\[
\|f_i\|_{k,j}
\leq\|f_1\|_1
\leq 1.
\]
Using this inequality together with \eqref{eight} and \eqref{two} we get, for each $i\in\{1,2\}$,
\begin{equation}\label{eleven}
\begin{aligned}
\Vvert f_i \Vvert^2
&=\sum_{k=0}^K 4^{-k}\sum_{j=1}^{2^k}\|f_i\|_{k,j}^2+\sum_{k=K+1}^\infty 4^{-k}\sum_{j=1}^{2^k}\|f_i\|_{k,j}^2\\
&\leq\Vvert f\Vvert^2+2^{-K}
\leq 1+\gamma.
\end{aligned}
\end{equation}
From \eqref{nine} we get
\begin{equation}\label{twelve}
\begin{aligned}
\Vvert f_1-f_2\Vvert^2
&\geq\sum_{k=0}^K 4^{-k}\sum_{j=1}^{2^k}\|f_1-f_2\|_{k,j}^2
=4\sum_{k=0}^K 4^{-k}\sum_{j=1}^{2^k}\|f\|_{k,j}^2\\
&=4\biggl(\Vvert f\Vvert^2-\sum_{k=K+1}^\infty 4^{-k}\sum_{j=1}^{2^k}\|f\|_{k,j}^2\biggr)\\
&\geq4\bigl(\Vvert f\Vvert^2 -2^{-K}\bigr
)=4\bigl(1-2^{-K}\bigr)
\geq 4(1-\gamma),
\end{aligned}
\end{equation}
where we invoked \eqref{two} in the last step.

For every $j\in\{1,\dotsc,2^K\}$,
\[
b_j
\leq \|f\|_\infty\,\lambda(I_j^K)
=2^{-K}\|f\|_\infty
\quad\text{and}\quad
c_j
\leq \|f\|_\infty\,\lambda(I_j^K)
=2^{-K}\|f\|_\infty.
\]
It follows that, for each $i\in\{1,2\}$,
\begin{equation}\label{four}
\|f_i\|_\infty
\leq 2^{K+2}\,2^{-K}\|f\|_\infty
=4\,\|f\|_\infty.
\end{equation}

Whenever $i\in\{1,2\}$ and $j\in\{1,\dotsc,2^K\}$, by \eqref{five} with $k=K$,
\[
\langle f_i-f,{\bf 1}_{I_j^K}\rangle
=\int_{I^K_j}(f_i-f)\,d\lambda=0,
\]
thus, for every $l\in\{1,\dots,m\}$,
keeping in mind that $s_l\in\linspan\bigr\{{\bf 1}_{I_j^K}\colon j\in\{1,2,\dotsc 2^K\}\bigr\}$,
one has $\langle f_i-f,s_l\rangle=0$.
It follows that, whenever $i\in\{1,2\}$ and $l\in\{1,\dotsc,m\}$, by \eqref{four} and \eqref{three},
\begin{equation}\label{fourteen}
\begin{aligned}
|\langle f_i-f,h_l\rangle|
&=|\langle f_i-f,h_l-s_l\rangle|
=\left|\int_{[0,1]} (f_i-f)\,(h_l-s_l)\,d\lambda\right|\\
&\leq\int_{[0,1]} |f_i-f|\,|h_l-s_l|\,d\lambda\\
&\leq\int_{[0,1]} (|f_i|+|f|)\, |h_l-s_l|\,d\lambda\\
&\leq\bigl(\|f_i\|_\infty+\|f\|_\infty\bigr)\,\|h_l-s_l\|_1\\
&\leq5\|f\|_\infty \,\|h_l-s_l\|_1\\
&<5\|f\|_\infty\, \gamma.
\end{aligned}
\end{equation}

For each $i\in\{1,2\}$, defining $g_i:=(1-\gamma)f_i$, we have $g_i\in U$ because, by~\eqref{eleven},
\begin{equation*}
\Vvert g_i\Vvert
=(1-\gamma)\Vvert f_i\Vvert\leq (1-\gamma)(1+\gamma)^{\frac{1}{2}}
=(1-\gamma)^{\frac{1}{2}}(1-\gamma^2)^{\frac{1}{2}}
<1
\end{equation*}
and, since $g_i-f_i=(1-\gamma)f_i-f_i=-\gamma f_i$, for every $l\in\{1,\dotsc,m\}$,
by \eqref{fourteen}, \eqref{ten}, and \eqref{one},
\begin{align*}
|\langle g_i-f,h_l|&\leq|\langle g_i-f_i,h_l|+|\langle f_i-f,h_l|\\
&<\|g_i-f_i\|_1\cdot\|h_l\|_\infty+5\|f\|_\infty\gamma\\
&=\gamma \|f_i\|_1+5\gamma\|f\|_\infty
=(5\|f\|_\infty+\|f_i\|_1)\gamma\\
&\leq(5\|f\|_\infty+1)\gamma
<\delta.
\end{align*}
It remains to observe that, by \eqref{twelve} and \eqref{one},
\[
\Vvert g_1-g_2\Vvert
=(1-\gamma)\Vvert f_1-f_2\Vvert
\geq 2(1-\gamma)^{\frac{3}{2}}
>2-\varepsilon.
\]
\end{proof}

The following proposition is the remaining portion of Theorem \ref{thm: (L_1[0,1],|||.|||) is wMLUR and has the D2P}.

\begin{prop}\label{prop:Rommet_er_wMLUR}
The Banach space $X=(L_1[0,1],\Vvert\cdot\Vvert)$ is wMLUR.
\end{prop}

We start with a lemma.

\begin{lem}\label{claim:Rommet_er_rotund}
The space $X$ is strictly convex.
\end{lem}
\begin{proof}
Let $f,g\in L_1[0,1]$ with $\Vvert f\Vvert=\Vvert g\Vvert=\Vvert\frac{f+g}{2}\Vvert=1$.
By a standard convexity argument (see, e.g., \cite[Fact II.2.3, (ii)]{DGZ}),
whenever $k\in\{0\}\cup\N$ and $j\in\{1,\dotsc,2^k\}$,
we have $\|f\|_{k,j}=\|g\|_{k,j}=\|\frac{f+g}{2}\|_{k,j}$ or, in other words,
\[
\int_{I_j^k}|f|\,d\lambda
=\int_{I_j^k}|g|\,d\lambda=\int_{I_j^k}\left|\frac{f+g}{2}\right|\,d\lambda.
\]
By applying Fact~\ref{fact2.2}, it follows that, for every $A\in\mathcal{L}$,
\begin{equation}\label{eq: int_A |f| dlambda = int_A |g| dlambda=int_A |f+g|/2 dlambda}
\int_A|f|\,d\lambda=\int_A|g|\,d\lambda=\int_A\left|\frac{f+g}{2}\right|\,d\lambda;
\end{equation}
thus $|f|=|g|=|\frac{f+g}{2}|$ a.e.
It follows that $f=g$ a.e.
Indeed, suppose for contradiction that $\lambda(A)>0$ where $A:=\{t\in[0,1]\colon f(t)\not=g(t)\}$.
Since $|f|=|g|$ a.e., we must have $f=-g$ a.e.~in~$A$,
therefore also $|f|\not=0$ a.e.~in~$A$ (and $|g|\not=0$ a.e.~in $A$),
hence $\int_A|f|\,d\lambda=\int_A|g|\,d\lambda>0$.
We now have
\begin{align*}
\int_A|f+g|\,d\lambda
&=0
<\int_A|f|\,d\lambda+\int_A|g|\,d\lambda.
\end{align*}
This contradicts \eqref{eq: int_A |f| dlambda = int_A |g| dlambda=int_A |f+g|/2 dlambda}.
\end{proof}

\begin{proof}[Proof of Proposition~\ref{prop:Rommet_er_wMLUR}]
Let $f\in L_1[0,1]$ with $\Vvert f\Vvert=1$ and let $(f_n)_{n=1}^\infty$ be a sequence in $L_1[0,1]$
such that $\Vvert f\pm f_n\Vvert\xrightarrow[n\to\infty]{}1$.
We must show that $f_n\xrightarrow[n\to\infty]{}0$ weakly in $L_1[0,1]$.
To this end, first observe that the sequence $(f_n)_{n=1}^\infty$ is equi-integrable (or uniformly integrable).
Indeed, suppose for contradiction that this is not the case.
Then there are a real number $\eps>\nobreak0$, a subsequence $(g_n)_{n=1}^\infty$ of $(f_n)_{n=1}^\infty$,
and a sequence $(A_n)_{n=1}^\infty$ in $\mathcal{L}$ with $\lambda(A_n)\xrightarrow[n\to\infty]{}0$
such that $\int_{A_n}|g_n|\,d\lambda\geq3\eps$ for every $n\in\N$.
Since the indefinite integral $\int_{(\cdot)}|f|\,d\lambda$ is absolutely continuous with respect to $\lambda$,
there is a real number $\delta>0$ such that $\int_D|f|\,d\lambda<\eps$
whenever $\lambda(D)<\delta$.
Choose $N\in\N$ such that $\lambda(A_n)<\delta$ for every $n\geq N$.
Since $\Vvert f\pm f_n\Vvert\xrightarrow[n\to\infty]{}1$,
by a standard convexity argument (see, e.g., \cite[Fact II.2.3, (ii)]{DGZ}),
whenever $k\in\{0\}\cup\N$ and $j\in\{1,\dotsc,2^k\}$,
we have $\|f\pm f_n\|_{k,j}\xrightarrow[n\to\infty]{}\|f\|_{k,j}$
and thus also 
\[
\frac{1}{2}\bigl(\|f+f_n\|_{k,j}+\|f-f_n\|_{k,j}\bigr)\xrightarrow[n\to\infty]{}\|f\|_{k,j}.
\]
On the other hand, whenever $n\in\N$, $n\geq N$, we have (recall that $I^0_1=[0,1)$)
\begin{align*}
\frac{1}{2}\bigl(\|(f+g_n\|_{0,1}+\|f-g_n\|_{0,1}\bigr)
&=\frac{1}{2}\biggl(\int_{I^0_1}|f+g_n|\,d\lambda+\int_{I^0_1}|f-g_n|\,d\lambda\biggr)\\
&=\frac{1}{2}\biggl(\int_{I^0_1\setminus A_n}|f+g_n|\,d\lambda+\int_{I^0_1\setminus A_n}|f-g_n|\,d\lambda\biggr)\\
&\quad\quad +\frac{1}{2}\biggl(\int_{A_n}|f+g_n|\,d\lambda+\int_{A_n}|f-g_n|\,d\lambda\biggr)\\
&\geq\int_{I^0_1\setminus A_n}|f|\,d\lambda+\int_{A_n}|g_n|\,d\lambda-\int_{A_n}|f|\,d\lambda\\
&=\|f\|_{0,1}+\int_{A_n}|g_n|\,d\lambda-2\int_{A_n}|f|\,d\lambda\\
&>\|f\|_{0,1}+3\eps-2\eps
=\|f\|_{0,1}+\eps,
\end{align*}
a contradiction.
Thus the sequence $(f_n)_{n=1}^\infty$ is equi-integrable.

Since the sequence $(f_n)_{n=1}^\infty$ is bounded,
by the characterisation due to Dunford and Pettis (see, e.g., \cite[page 109, Theorem 5.2.9]{AK}),
this sequence is relatively weakly compact.
Thus, by passing to a subsequence, we may assume that $f_n\xrightarrow[n\to\infty]{}g$ weakly for some $g\in L_1[0,1]$
(because it suffices to show that every subsequence of $(f_n)_{n=1}^\infty$
has a further subsequence converging weakly to $0$).
We have
\[
\Vvert f\Vvert=1=\lim_{n\to\infty}\Vvert f\pm f_n\Vvert\geq \Vvert f\pm g\Vvert.
\]
Since the space $X$ is strictly convex, $f$ is an extreme point of $B_X$, therefore $g=0$.
\end{proof}

\section{Proof of Theorem \ref{thm: norm-one projection with finite-co-dimensional range, URED, and D2P}}

\begin{proof}[Proof of Theorem \ref{thm: norm-one projection with finite-co-dimensional range, URED, and D2P}]
We follow the basic idea of the proof of \cite[Proposition 2.10]{ALNT}.
Assume that $X$ has the D$2$P and let $0<\delta<1$.
Let $z\in\ker P$ with $\|z\|=1-\delta$ and let $(\varepsilon_n)_{n=1}^\infty$ be a sequence of positive numbers tending to zero.
The assertion follows from the following claim.

\begin{Claim}
There exist sequences $(x_n)_{n=1}^\infty$ in $B_X\cap\ran P$ and $(\xs_n)_{n=1}^\infty$ in $B_{\Xs}$ such that
\begin{enumerate}
\item [\textup{(1)}] $\|z+x_n\|<1$ for every $n\in\mathbb{N}$;
\item [\textup{(2)}] $\xs_n(x_m)=\xs_n(x_n)>1-\varepsilon_n$ whenever $m,n\in\mathbb{N}$ with $m\geq n$.
\end{enumerate}
\end{Claim}

\noindent%
Indeed, suppose that Claim has been proven.
For every $n\in\N$, defining $y_n:=z+x_n$, one has $x_n,y_n\in B_X$.
Whenever $m,n\in\mathbb{N}$ with $m\geq n$, one has
(keeping in mind that $z\in\ker P$ and $x_m\in\ran P$)
\begin{equation}\label{eq4}
\left\|\tfrac{1}{2}z+x_m\right\|
\geq(P^\ast\xs_n)\left(\tfrac{1}{2}z+x_m\right)
=\xs_n(x_m)
>1-\varepsilon_n;
\end{equation}
thus
\[
\|x_n+y_n\|=\|2x_n+z\|\xrightarrow[n\to\infty]{}2.
\]
On the other hand, $x_n-y_n=-z$ for every $n\in\N$.
Since $-z\not=0$, it follows that $X$ is not URED.

By the Banach--Alaoglu theorem, the sequence $(x_n)_{n=1}^\infty$ has a subnet
weak$^\ast$ converging in $\Xss$ to some $\xss\in B_{\Xss}$.
The corresponding subnet of $(z+x_n)_{n=1}^\infty$ converges to $z+\xss$ weak$^\ast$ in $\Xss$.
From (1) it follows that $\|z+\xss\|\leq1$,
thus the line segment $\{\lambda z+\xss\colon \lambda\in[0,1]\}$ (of length $\|z\|=1-\delta$)
is in $B_{\Xss}$. In order to see that this line segment is, in fact, in the unit sphere of $\Xss$,
it suffices to observe that, by \eqref{eq4}, for every $n\in\N$,
\[
\bigl\|\tfrac{1}{2}z+\xss\bigr\|
\geq\bigl(\tfrac{1}{2}z+\xss\bigr)(P^\ast\xs_n)
\geq1-\eps_n,
\]
and thus $\bigl\|\tfrac{1}{2}z+\xss\bigr\|=1$.

\medskip
It remains to prove Claim.
To this end, define $x_0:=0\in \ran P$, $\xs_0:=0\in\Xs$, and $\varepsilon_0:=2$.
Obviously $\|z+x_0\|<1$ and $\xs_0(x_0)=0>1-\varepsilon_0$.
The desired sequences $(x_n)_{n=1}^\infty$ and $(\xs_n)_{n=1}^\infty$ will be defined recursively as follows.

Suppose that, for some $n\in\{0\}\cup\N$,
we have defined $x_0,\dotsc,x_n \in B_X\cap\ran P$ and $\xs_0,\dotsc,\xs_n\in B_{\Xs}$
so that $\|z+x_i\|<1$ for all $i\in\{0,\dotsc,n\}$
and $\xs_i(x_j)=\xs_i(x_i)>1-\varepsilon_i$ whenever $i,j\in\{0,\dotsc, n\}$ with $j\geq i$,
but $x_{n+1}$ and $\xs_{n+1}$ have not been defined yet.
By \cite[Proposition 2.7]{ALNT}
there are $y_1,y_2\in \ran P\cap\bigcap_{i=0}^n \ker \xs_i$
such that $\|z+x_n+y_k\|<1$, $k=1,2$, and $\|y_1-y_2\|>2-\varepsilon_{n+1}$. Since
\[
\|x_n+y_k\|
=\|P(z+x_n+y_k)\|
\leq \|P\|\,\|z+x_n+y_k\|<1,
\quad k=1,2,
\]
and
\[
\|(x_n+y_1)-(x_n+y_2)\|
=\|y_1-y_2\|
>2-\varepsilon_{n+1},
\]
it follows that $1>\|x_n+y_k\|>1-\varepsilon_{n+1}$, $k=1,2$.
Therefore, defining $x_{n+1}:=x_n+y_1$, we have $x_{n+1}\in B_X\cap\ran P$
with $\|z+x_{n+1}\|=\|z+x_n+y_1\|<1$.
Since $\|x_{n+1}\|=\|x_n+y_1\|>1-\varepsilon_{n+1}$,
we can pick an $\xs_{n+1}\in B_{\Xs}$ with $\xs_{n+1}(x_{n+1})>1-\varepsilon_{n+1}$.
It remains to observe that $\xs_i(x_{n+1})=\xs_i(x_n)=\xs_i(x_i)$ for every $i\in\{0,\dotsc,n\}$
because $x_{n+1}-x_n=y_1\in \bigcap_{i=0}^n\ker\xs_i$.
\end{proof}

\section{Notes and remarks}

\begin{rem}\label{remark1}
By invoking Choquet's lemma as in \cite[Proposition~1.3]{AHNTT},
in Proposition~\ref{prop:rommet_har_D2P} it suffices, in fact, to prove that
every slice of the closed unit ball of $X$ has diameter $2$.
However, the proof does not become easier.
\end{rem}

\begin{rem}\label{remark2}
We do not know if it is possible to renorm $L_1[0,1]$ in such a way that it is MLUR and has the D$2$P.
For the renorming done here, it can be seen that the unit ball contains no strongly extreme points.
\end{rem}

\begin{rem}\label{remark3}
We do not know whether there exists a Banach space $X$ with the D$2$P
such that its bidual $\Xss$ is strictly convex
(cf. Theorem \ref{thm: norm-one projection with finite-co-dimensional range, URED, and D2P}).
\end{rem}

Concerning Remark \ref{remark3}, the analogous question for the \emph{strong diameter~$2$ property}
was answered in the negative in \cite{ALNT}
(see Theorem \ref{thm: (a) X* is strictly convex => X* fails the w*-SD2P; (b) ...} below).
Recall the relevant definitions.
Throughout, until the end of the paper, $X$ will be a Banach space.

\begin{definition}
The space $X$ (or its norm) has the \emph{strong diameter~$2$ property} (SD$2$P for short)
if every finite convex combination of slices of $B_X$ has diameter $2$.
If $X$ is a dual space and every finite convex combination of weak$^\ast$ slices
of $B_X$ has diameter $2$, then $X$ has the \emph{weak$^\ast$ strong diameter~$2$ property} ($w^\ast$-SD$2$P for short).
\end{definition}

\noindent%
Using Goldstine’s theorem, it is not difficult to see that \emph{$X$ has the SD$2$P if and only if the bidual $\Xss$ has the $w^\ast$-SD$2$P.}

\begin{definition}
The space $X$ (or its norm) is \emph{octahedral} (OH for short) if, for every finite dimensional subspace $E$ of $X$
and every $\eps>0$, there is a $y\in S_X$  such that
\[
\|x+\alpha y\|\geq(1-\eps)(\|x\|+|\alpha|)
\quad\text{for every $x\in E$ and every $\alpha\in\R$.}
\]
\end{definition}

\noindent%
There is a duality between octahedrality and the SD$2$P:
by a result of Godefroy \cite[Remark II.5, 2]{G}
(see \cite[Theorem 2.1]{BG&L-P&RZ} for a proof)
\emph{$X$ is OH if and only if $\Xs$ has the $w^\ast$-SD$2$P.}

\begin{definition}
The space $X$ contains an \emph{asymptotically isometric copy of $\ell_1$} if there are a sequence $(x_k)_{k=1}^\infty$ in $X$
and a sequence $(\delta_k)_{k=1}^\infty$ of real numbers in the interval $(0,1)$ with $\delta_k\xrightarrow[k\to\infty]{}0$
such that, whenever $m\in\mathbb{N}$ and $\alpha_1,\dotsc,\alpha_m\in\R$, one has
\begin{equation}\label{eq: asympotically isomorphic to ell_1}
\sum_{k=1}^{m}(1-\delta_k)|\alpha_k|
\leq\biggl\|\sum_{k=1}^m \alpha_k x_k\biggr\|
\leq\sum_{k=1}^{m}|\alpha_k|.
\end{equation}
\end{definition}

The following result was proven in \cite{ALNT}.

\begin{thm}\label{thm: (a) X* is strictly convex => X* fails the w*-SD2P; (b) ...}
\begin{itemize}
\item[\textup{(a)}]
If $\Xs$ is strictly convex, then $\Xs$ fails the $w^\ast$-SD$2$P.
\item[\textup{(b)}]
If $\Xss$ is strictly convex, then $X$ fails the SD$2$P.
\end{itemize}
\end{thm}

\noindent%
By the aforementioned duality between octahedrality and the SD$2$P,
Theorem \ref{thm: (a) X* is strictly convex => X* fails the w*-SD2P; (b) ...}
is an immediate consequence of the following theorem.

\begin{theorem}[see the proof of {\cite[Theorem 2.5]{ALNT}} and the sentence following it]%
\label{thm: X* is R => X is not OH}
If $\Xs$ is strictly convex, then $X$ is not OH.
\end{theorem}

\noindent%
In \cite{ALNT}, Theorem \ref{thm: X* is R => X is not OH}
was observed as a(n obvious) consequence of the following two results.

\begin{proposition}[see {\cite[Lemma 2.3]{ALNT}}]\label{prop: X is OH => X contains an asympt. ell_1}
If $X$ is OH, then it contains an asymptotically isometric copy of $\ell_1$.
\end{proposition}

\begin{theorem}\label{thm: X contains an asympt. copy of ell_1 => X^ast is not R}
If $X$ contains an asymptotically isometric copy of $\ell_1$, then $\Xs$ is not strictly convex.
\end{theorem}

\noindent%
The proof of Theorem \ref{thm: X contains an asympt. copy of ell_1 => X^ast is not R} in \cite{ALNT}
(see the proof of {\cite[Theorem 2.5]{ALNT}} and the sentence following it) relies on \cite[Theorem 2]{DGH}:
if $X$ contains an asymptotically isometric copy of $\ell_1$, then, by \cite[Theorem 2]{DGH},
the dual space $\Xs$ contains an isometric copy of $L_1[0,1]$; since $L_1[0,1]$ is not strictly convex, neither $\Xs$ is
(also, since $L_1[0,1]$ is not smooth, neither $\Xs$ is).

The next proposition
provides a self-contained proof of Theorem \ref{thm: X contains an asympt. copy of ell_1 => X^ast is not R}
(without the help of \cite[Theorem 2]{DGH}).

\begin{proposition}
Suppose that $X$ contains an asymptotically isometric copy of $\ell_1$. Then
\begin{itemize}
\item[\textup{(a)}]
the dual unit sphere $S_{\Xs}$ contains a line segment of length $2$;
\item[\textup{(b)}]
the dual space $\Xs$ is not smooth.
\end{itemize}
\end{proposition}
\begin{proof}
Since $X$ contains an asymptotically isometric copy of $\ell_1$, there are a sequence $(x_k)_{k=1}^\infty$ in $X$
and a sequence $(\delta_k)_{k=1}^\infty$ of real numbers in the interval $(0,1)$ with $\delta_k\xrightarrow[k\to\infty]{}0$
satisfying \eqref{eq: asympotically isomorphic to ell_1} whenever $m\in\mathbb{N}$ and $\alpha_1,\dotsc,\alpha_m\in\R$.

(a).
For every $m\in\N$, let $f_m,g_m\colon\linspan\{x_1,\dotsc,x_{2m}\}\to\R$ be the linear functionals satisfying
$f_m(x_{2i-1})=1-\delta_{2i-1}$, $g_m(x_{2i-1})=-1+\delta_{2i-1}$, and $f_m(x_{2i})=g_m(x_{2i})=1-\delta_{2i}$
for every $i\in\{1,\dotsc,m\}$, and let $\xs_m\in\Xs$ and $\ys_m\in\Xs$ be some norm-preserving extensions of,
respectively, $f_m$ and $g_m$. Letting $\xs$ and $\ys$ be any  weak$^\ast$ limit points of, respectively,
the sequences $(\xs_m)_{m=1}^\infty$ and $(\ys_m)_{m=1}^\infty$,
one has $\|\xs\|=\|\ys\|=\bigl\|\frac{1}{2}(\xs+\ys)\bigr\|=1$
and $\|\xs-\ys\|=2$, thus the line segment connecting $\xs$ and $\ys$ is in $S_{\Xs}$ and has length $2$.
Indeed, whenever $m\in\N$ and $x=\sum_{k=1}^{2m}\alpha_k x_k\in\linspan\{x_1,\dotsc,x_{2m}\}$, one has
\begin{align*}
|f_m(x)|
&=\biggl|\sum_{k=1}^{2m}\alpha_k f_m(x_k)\biggr|
\leq\sum_{k=1}^{2m}|\alpha_k|\,|f_m(x_k)|
=\sum_{k=1}^{2m}|\alpha_k|\,|1-\delta_k|\\
&\leq\biggl\|\sum_{k=1}^{2m}\alpha_k x_k\biggr\|
=\|x\|
\end{align*}
and, similarly, $|g_m(x)|\leq\|x\|$, thus $\|\xs_m\|=\|f_m\|\leq1$ and  $\|\ys_m\|=\|g_m\|\leq1$,
and it follows that $\|\xs\|\leq1$ and $\|\ys\|\leq1$.

On the other hand, whenever $i,m\in\N$ with $m\geq i$, one has
\begin{alignat*}{2}
\xs_m(x_{2i})&=f_m(x_{2i})=1-\delta_{2i},&\qquad
\ys_m(x_{2i})&=g_m(x_{2i})=1-\delta_{2i},\\
\xs_m(x_{2i-1})&=f_m(x_{2i-1})=1-\delta_{2i-1},&\qquad
\ys_m(x_{2i-1})&=g_m(x_{2i-1})=-1+\delta_{2i-1},
\end{alignat*}
thus
\[
\xs(x_{2i})=\ys(x_{2i})=\tfrac{1}{2}(\xs+\ys)(x_{2i})=1-\delta_{2i}
\]
and
\[
\xs(x_{2i-1})=1-\delta_{2i-1}
\qquad\text{and }\qquad
\ys(x_{2i-1})=-1+\delta_{2i-1},
\]
whence  $\|\xs\|=\|\ys\|=\bigl\|\frac{1}{2}(\xs+\ys)\bigr\|=1$ and $\|\xs-\ys\|=2$
(here we use that $\|x_k\|\leq1$ for every $k\in\N$
by the second inequality in \eqref{eq: asympotically isomorphic to ell_1}).


\medskip
(b).
Letting $\xss$ and $\yss$ be any  weak$^\ast$ limit points in $\Xss$ of, respectively,
the sequences $(x_{2i-1})_{i=1}^\infty$ and $(x_{2i})_{i=1}^\infty$, one has $\|\xss\|=\|\yss\|=1$
because $\xss(\xs)=\yss(\xs)=1$ (we work with the elements $\xs$ and $\ys$ defined in the proof of (a)).
To see that the dual space $\Xs$ is not smooth, it suffices to observe that $\xss\not=\yss$
(because $\xss(\ys)=-1$ and $\yss(\ys)=1$, and thus $\|\xss-\yss\|=2$).
\end{proof}

For the sake of completeness,
we conclude by proving Proposition \ref{prop: X is OH => X contains an asympt. ell_1}
by a slightly simpler argument than that in \cite{ALNT}.

\begin{proof}[Proof of Proposition \ref{prop: X is OH => X contains an asympt. ell_1}]
Assume that $X$ is octahedral.
Let $(\delta_k)_{k=1}^\infty$ be a sequence in the interval~$(0,1)$.
It suffices to find a sequence $(x_k)_{k=1}^\infty$ in $S_X$ such that,
whenever $m\in\mathbb{N}$ and $\alpha_1,\dotsc,\alpha_m\in\R$, 
one has \eqref{eq: asympotically isomorphic to ell_1}.
To this end, pick real numbers $\eps_i>0$, $i=1,2,\dotsc$, so that
\[
\prod_{i=k}^m(1-\eps_i)>1-\delta_k
\quad\text{for all $m,k\in\N$ with $m\geq k$.}
\]
Since $X$ is octahedral, one can pick elements $x_k\in S_X$, $k=1,2,\dotsc$, such that
for every $m\in\N$ with $m\geq2$, every $x\in\linspan\{x_1,\dots,x_{m-1}\}$, and every $\alpha\in\R$, one has
\[
\|x+\alpha x_m\|
\geq(1-\eps_{m})\bigl(\|x\|+|\alpha|\bigr).
\]
The condition \eqref{eq: asympotically isomorphic to ell_1} holds
whenever $m\in\mathbb{N}$ and $\alpha_1,\dotsc,\alpha_m\in\R$.
\end{proof}

\providecommand{\bysame}{\leavevmode\hbox to3em{\hrulefill}\thinspace}
\providecommand{\MR}{\relax\ifhmode\unskip\space\fi MR }
\providecommand{\MRhref}[2]{%
  \href{http://www.ams.org/mathscinet-getitem?mr=#1}{#2}
}
\providecommand{\href}[2]{#2}

\end{document}